\documentclass[graybox]{svmult}

\usepackage{mathptmx}       
\usepackage{helvet}         
\usepackage{courier}        
\usepackage{type1cm}        
\usepackage{makeidx}         
\usepackage{graphicx}        
\usepackage{multicol}        
\usepackage[bottom]{footmisc}

\usepackage{amsmath}
\usepackage{amssymb}
\usepackage{amsfonts}

\makeindex             

\begin{document}

\title*{Computation of expectations by Markov chain Monte Carlo methods}


\author{Erich Novak and Daniel Rudolf}


\institute{Erich Novak \at Friedrich Schiller University Jena, 
Mathematical Institute,
Ernst-Abbe-Platz 2, 
D-07743 Jena, Germany, \email{erich.novak@uni-jena.de}
\and Daniel Rudolf \at Friedrich Schiller University Jena, 
Mathematical Institute,
Ernst-Abbe-Platz 2, 
D-07743 Jena, Germany, \email{daniel.rudolf@uni-jena.de}}


\maketitle

\abstract{Markov chain Monte Carlo\index{Markov chain Monte Carlo} 
(MCMC) methods are a very versatile and widely 
used tool to compute integrals and expectations. 
In this short survey we focus on 
error bounds, rules for choosing the burn in, 
high dimensional problems and tractability versus 
curse of dimension.}

\section{Motivation}    \label{novak_rudolf sec: motivation}

Consider the following example.
We want to compute
\[
    \mathbb{E}_G(f)  = \frac{1}{{\rm  vol}_d(G)} \int_G f(x)\, \text{\rm d} x,
\]
where $f$ belongs to some class of functions 
and $G$ belongs to some class of sets. 
We assume that 
$G\subset \mathbb{R}^d$
is measurable  with 
$0<{\rm  vol}_d(G)<\infty$, where ${\rm  vol}_d$ denotes 
the Lebesgue measure. 
Thus, we want to compute the expected value of $f$
with respect to the uniform distribution on $G$. 

The input $(f, G)$ is given by an oracle: 
For $x \in G$ we can compute $f(x)$ 
and
$G$ is given by a membership oracle, i.e. 
we are able to check whether 
any $x\in \mathbb{R}^d$
is in $G$ or not. 
We always assume that $G$ is convex and will work with the 
class 
\begin{equation}      \label{novak_rudolf classGrd} 
\mathcal{G}_{r,d} = \{ G \subset \mathbb{R}^d 
\colon G \text{\,is convex},\; B_d \subset G \subset  rB_d\}, 
\end{equation} 
where 
$r\geq 1$
and 
$rB_d=\{x\in \mathbb{R}^d \colon \vert x \vert \leq r \}$ 
is the Euclidean ball with radius $r$.

A first approach 
might be a simple acceptance/rejection method.
The idea is to generate a point in $rB_d$
according to the uniform distribution
and if it is in $G$ it is accepted, otherwise 
it is rejected. 
If $x_1,\dots,x_n \in G$ are the accepted points then we output the
mean value of the $f(x_i)$.
However, this method does not work reasonably
since 
the acceptance probability
can be extremely small, it can be 
$r^{-d}$. 

It seems that all known efficient algorithms for this problem 
use Markov chains. 
The idea is to find a sampling procedure
that approximates a sample with respect to the uniform distribution in $G$. 
More precisely, we run a Markov chain to 
approximate the uniform distribution for any 
$G\in\mathcal{G}_{r,d}$.
Let $X_1,X_2,\dots,X_{n+n_0}$ be the first $n+n_0$ 
steps of such a Markov chain.
Then 
\[
S_{n,n_0}(f,G)= \frac{1}{n} \sum_{j=1}^n f(X_{j+n_0})
\]
is an approximation of $\mathbb{E}_G(f)$.
The additional parameter $n_0$ is called burn-in and, roughly spoken,
is the number of steps of the Markov chain to 
get close to the uniform distribution.

\section{Approximation of expectations by MCMC}

\subsection{Preliminaries}

We provide the basics of Markov chains. For further reading we refer to
the paper \cite{novak rudolf:RoRo04} of Roberts and Rosenthal 
which surveys various results about Markov chains
on general state spaces.

A Markov chain is a sequence of random variables 
$(X_n)_{n\in\mathbb{N}}$ which satisfies the Markov property. 
For $i\in \mathbb{N}$, the conditional
distribution of $X_{i+1}$ depends only 
on $X_i$ and not on $(X_1,\dots,X_{i-1})$, 
\[
\mathbb{P}(X_{i+1}\in A \mid X_1,\dots,X_i) 
= \mathbb{P}(X_{i+1}\in A\mid X_i).
\]
By $\mathcal{B}(G)$ we denote the Borel $\sigma$-algebra of $G$.
Let $\nu$ be a distribution on $(G,\mathcal{B}(G))$ and 
let $K \colon G \times \mathcal{B}(G) \to [0,1]$
be a \emph{transition kernel}\index{transition kernel}, i.e. $K(x,\cdot)$ is a 
probability measure for each $x \in G$ and
$K(\cdot,A)$ is a $\mathcal{B}(G)$-measurable real-valued function 
for each $A \in \mathcal{B}(G)$.
A transition kernel and a distribution $\nu$ give rise 
to a Markov chain $(X_n)_{n\in\mathbb{N}}$
in the following way. 
Assume that the distribution of $X_1$ is given by $\nu$. 
Then, for $i\geq 2$ and a given
$X_{i-1}=x_{i-1}$, we have $X_i$ with distribution $K(x_{i-1},\cdot)$, 
that is, for all $A\in\mathcal{B}(G)$,
the conditional probability that $X_i \in A$ is given by $K(x_{i-1},A)$.  
We call such a sequence of random variables a Markov chain 
with transition kernel $K$ and initial distribution $\nu$. 

In the whole paper we only consider Markov 
chains with \emph{reversible}\index{reversible} transition kernel, 
we assume that there exists a probability 
measure $\pi$ on $\mathcal{B}(G)$ such that
\[
\int_A K(x,B)\, \pi({\rm d}x) = 
\int_B K(x,A)\, \pi({\rm d}x), \quad A,B \in \mathcal{B}(G).
\]
In particular any such $\pi$ is a stationary distribution of $K$, i.e.,
\[
 \pi(A) = \int_G K(x,A) \, \pi({\rm d} x), \quad A\in\mathcal{B}(G).
\]

Further, the transition kernel induces an operator 
on functions and an operator on measures given by
\[
 P f(x) = \int_G f(y)\, K(x,{\rm d}x), 
\quad  
\text{and}
\quad
 \nu P(A) = \int_G K(x,A)\, \nu({\rm d}x),
\]
where $f$ is $\pi$-integrable and $\nu$ is absolutely 
continuous with respect to $\pi$.
One has
\[
\mathbb{E}[f(X_n) \mid X_1 = x] = P^{n-1} f(x)
\quad
\text{and}
\quad
\mathbb{P}_\nu(X_n\in A) = \nu P^{n-1} (A),
\]
for $x\in G$, $A \in \mathcal{B}(G)$ and $n\in \mathbb{N}$, 
where $\nu$  in $\mathbb{P}_\nu$
indicates that $X_1$ has distribution $\nu$.
By the reversibility with respect to $\pi$ we have
$
\frac{d(\nu P)}{d\pi}(x) = P (\frac{d \nu}{d \pi})(x),
$
where $\frac{d \nu}{d \pi}$ denotes the density of $\nu$ with respect to $\pi$.

Further, for $p\in[1,\infty)$ let $L_p=L_p(\pi)$ be the 
space of measurable functions $f\colon G \to \mathbb{R}$ which satisfy
\[
\Vert f \Vert _{p} = \left( \int_G \vert f(x) 
\vert^p \pi({\rm d} x) \right)^{1/p} < \infty.
\]
The operator $P \colon L_p \to L_p$ is linear and 
bounded and by the reversibility
$P\colon L_2 \to L_2$ is self-adjoint.

The goal is to quantify the speed of convergence, if it converges at all, 
of $\nu P^{n}$ to $\pi$ for increasing $n\in\mathbb{N}$. For this we use
the \emph{total variation distance}\index{total variation distance}
between two probability measures $\nu,\mu$ on $(G,\mathcal{B}(G))$
given by
\[
\Vert \nu-\mu \Vert_{\text{\rm tv}}
=\sup_{A\in\mathcal{B}(G)}\vert \nu(A)-\mu(A) \vert.
\]
It is helpful to consider the total variation distance as an $L_1$-norm, 
see for example \cite[Proposition~3, p.~28]{novak rudolf:RoRo04}.

\begin{lemma}  \label{novak_rudolf tv_present}
Assume the probability measures $\nu,\mu$ have 
densities $\frac{d \nu}{d \pi}, \frac{d \mu}{d \pi} \in L_1$, then
$
\Vert \nu-\mu \Vert_{\text{\rm tv}}
=\frac{1}{2}\left \Vert \frac{d\nu}{d\pi}-\frac{d\mu}{d\pi}\right \Vert_{1}.
$
\end{lemma}

Now we ask for an upper bound of 
$\left \Vert \nu P^n - \pi\right \Vert_{\text{\rm tv}}$.

\begin{lemma}  \label{novak_rudolf lem: est_tv}
Let $\nu$ 
be a probability measure on $(G,\mathcal{B}(G))$ 
with $\frac{d\nu}{d\pi}\in L_1$ 
and let
$
S(f) = \int_G f(x)\,\pi(\text{\rm d} x).
$
Then, for any $n\in\mathbb{N}$ holds
\[
\left \Vert \nu P^n - \pi \right \Vert_{\text{\rm tv}} 
\leq 	\left \Vert P^n-S \right \Vert_{L_1\to L_1} \frac{1}{2} 
\left \Vert \frac{d\nu}{d\pi}-1\right \Vert_{1} 
\leq \left \Vert P^n-S \right \Vert_{L_1\to L_1} 
\]
and
\[
\left \Vert \nu P^n - \pi \right \Vert _{\text{\rm tv}} \leq  
\left \Vert P^n-S \right \Vert_{L_2\to L_2} \frac{1}{2} 
\left \Vert \frac{d\nu}{d\pi}-1 \right \Vert_{2}.
\]
\end{lemma}

\begin{proof}
By Lemma~\ref{novak_rudolf tv_present}, 
by $P^n 1 = 1$ and by the reversibility, in particular
$
\frac{d(\nu P^n)}{d\pi}(x) = P^n (\frac{d \nu}{d \pi})(x),
$ we have 
\begin{align*}
2 \left \Vert \nu P^n - \pi \right \Vert _{\text{\rm tv}} &
= \left \Vert \frac{d(\nu P^n)}{d\pi} - 1 \right \Vert_{1}  
=   \left \Vert P^n\left(\frac{d\nu}{d\pi}-1\right)\right \Vert_{1} 
= \left \Vert (P^n-S)\left(\frac{d\nu}{d\pi}-1\right)\right \Vert_{1}.
\end{align*}
Note that the last equality comes from $S(\frac{d\nu}{d\pi}-1)=0$. 
\end{proof}

Observe that for $\nu = \pi$ the left-hand side 
and also the right-hand side of the estimates are zero.

Let us consider  $\left \Vert P^n-S \right \Vert_{L_2\to L_2}$. Because of the
reversibility with respect to $\pi$ we obtain the following, 
see for example \cite[Lemma~3.16, p.~45]{novak rudolf:Ru12}.

\begin{lemma}  \label{novak_rudolf lem: norm_spec_gap}
For $n\in \mathbb{N}$ we have
\[
\left \Vert P^n-S \right \Vert_{L_2 \to L_2 } = 
\left \Vert (P-S)^n \right \Vert_{L_2 \to L_2} 
= \left \Vert P-S \right \Vert_{L_2 \to L_2}^n.
\]
\end{lemma}

The last two lemmata motivate the following two convergence 
properties of transition kernels.

\begin{definition}
 [\index{$L_1$-exponential convergence}$L_1$-exponential convergence]
Let $\alpha\in[0,1)$ and $M\in(0,\infty)$. Then the transition kernel $K$
is $L_1$-exponentially convergent with $(\alpha,M)$ if
\begin{equation}  \label{novak_rudolf eq: L_1_exp_con}
\left \Vert P^n-S \right \Vert_{L_1 \to L_1} 
\leq \alpha^n M,\quad n\in\mathbb{N}.
 \end{equation}
A Markov chain with transition kernel $K$ 
is called $L_1$-exponentially convergent if there
exist an $\alpha\in[0,1)$ and $M\in(0,\infty)$ 
such that \eqref{novak_rudolf eq: L_1_exp_con} holds.
\end{definition}

\begin{definition}[\index{$L_2$-spectral gap}$L_2$-spectral gap]
 We say that a transition kernel 
 $K$ and its corresponding Markov operator $P$ have an $L_2$-spectral gap
 if
 \[
  {\rm gap}(P)=1-\left \Vert P-S\right \Vert_{L_2 \to L_2}>0.
  \] 
\end{definition}
If the transition kernel has an $L_2$-spectral gap, 
then by Lemma~\ref{novak_rudolf lem: est_tv} 
and Lemma~\ref{novak_rudolf lem: norm_spec_gap}  we have  that
\[
\left \Vert \nu P^n - \pi \right \Vert_{\text{\rm tv}} 
\leq (1-{\rm gap}(P))^n \left \Vert \frac{d\nu}{d\pi}-1\right \Vert_{2}.
\]

Next, we define other convergence properties 
which are based on the total variation distance. 

\begin{definition}[\index{uniform ergodicity}uniform ergodicity and
	      \index{geometric ergodicity}geometric ergodicity]
Let  $\alpha\in [0,1)$ and $M\colon G \to (0,\infty)$. 
Then the transition kernel $K$ is called 
\emph{geometrically ergodic with $(\alpha,M(x))$} if one has
for $\pi$-almost all $x\in G$ that
\begin{equation}  \label{novak_rudolf eq: pi_erg}
\left \Vert K^n(x,\cdot)-\pi\right \Vert_{\text{\rm tv}}\leq M(x)\, 
\alpha^n ,\quad n\in\mathbb{N}. 
\end{equation}
If the inequality $\eqref{novak_rudolf eq: pi_erg}$ holds with a bounded 
function $M(x)$, i.e. 
\[
\sup_{x\in G} M(x) \leq M' <\infty,
\]
then $K$ is called \emph{uniformly ergodic with $(\alpha,M')$}.
\end{definition}

Now we state several relations between the different properties.
Since we assume that the transition kernel is reversible with respect to $\pi$
we have the following:
\begin{equation}  \label{novak_rudolf eq: diagramm}
\begin{array}{ccc}
\mbox{uniformly ergodic} & 
\Longleftrightarrow & \mbox{$L_1$-exponentially convergent } \\
\mbox{with } (\alpha,M)     & 	 &  \mbox{with } (\alpha,2M) \\[1ex]
\mathbin{\text{\rotatebox[origin=c]{90}{$\Longleftarrow$}}}  & & 
\mathbin{\text{\rotatebox[origin=c]{90}{$\Longleftarrow$}}} \\[1ex]
\mbox{geometrically ergodic} &   & \mbox{$L_2$-spectral gap $\geq$}\\
\mbox{with } (\alpha,M(x))  &  			& 	1-\alpha.
\end{array}
\end{equation}
The fact that uniform ergodicity implies geometric ergodicity is obvious.
For the proofs of the other relations and further details we refer to 
\cite[Proposition~3.23, Proposition~3.24]{novak rudolf:Ru12}.
Further, if the transition kernel is $\varphi$-irreducible, 
for details we refer 
to \cite{novak rudolf:RoRo97} and \cite{novak rudolf:RoTw01},
then 
\begin{equation}  
\begin{array}{ccc}
\mbox{geometrically ergodic} &\quad 
\Longleftrightarrow \quad & \mbox{$L_2$-spectral gap $\geq$}\\
\mbox{with } (\alpha,M(x))  &  			& 	1-\alpha.
\end{array}
\end{equation}

\subsection{Mean square error bounds of MCMC}

The goal is to compute
\[
 S(f) = \int_G f(x)\, \pi(\text{\rm d} x).
\]
We use an average of a finite Markov chain sample 
as approximation of the mean, i.e. we approximate $S(f)$ by
\[
  S_{n,n_0}(f) = \frac{1}{n} \sum_{j=1}^n f(X_{j+n_0}).
\]
The number $n$ determines the number of
function evaluations of $f$.
The number $n_0$ is the \emph{burn-in}\index{burn-in} 
or \emph{warm up} time. Intuitively, 
it is the number of steps of the 
Markov chain to get close to the stationary distribution $\pi$.

We study the mean square error of $S_{n,n_0}$, 
given by
\[
e_\nu(S_{n,n_0},f) = \left( \mathbb{E}_{\nu,K} 
\vert S_{n,n_0}(f)-S(f) \vert \right)^{1/2},
\]
where $\nu$ and $K$ indicate the initial distribution and
transition kernel.	
We start with the case $\nu=\pi$, where the initial
distribution is the stationary distribution.

\begin{lemma}  \label{novak_rudolf thm: err_bound_stat}
Let $(X_n)_{n\in\mathbb{N}}$ be a Markov chain with transition kernel $K$
and initial distribution $\pi$. We define 
\[
\Lambda = \sup\{ \alpha \colon \alpha \in {\rm {s}pec}(P-S) \},
\]
where ${\rm {s}pec}(P-S)$ denotes the 
spectrum of the operator $P-S\colon L_2 \to L_2$, 
and assume that $\Lambda<1$.
Then
\[
\sup_{\left \Vert f\right \Vert_{2}\leq 1} 
e_\pi(S_{n,n_0},f)^2 \leq \frac{2}{n(1-\Lambda)}.
\]
\end{lemma}

For a proof of this result we 
refer to \cite[Corollary~3.27]{novak rudolf:Ru12}.
Let us discuss the assumptions and implications of 
Lemma~\ref{novak_rudolf thm: err_bound_stat}.
First, note that for the simple Monte Carlo method we have $\Lambda=0$.
In this case we get (up to a constant of 2) what we would expect. 
Further, note that ${\rm gap}(P)
=1-\left \Vert P-S\right \Vert_{L_2 \to L_2}$ and
\[
\left \Vert P-S\right \Vert_{L_2 \to L_2} = 
\sup \{\vert \alpha \vert \colon \alpha \in {\rm {s}pec}(P-S)\},
\]
so that ${\rm gap}(P)\leq 1-\Lambda$.
This also implies that if $P\colon L_2 \to L_2$ 
is positive semidefinite we obtain ${\rm gap}(P) = 1-\Lambda$.
Thus, whenever we have a lower bound for the spectral gap 
we can apply Lemma~\ref{novak_rudolf thm: err_bound_stat} 
and can replace $1-\Lambda$ by ${\rm gap}(P)$. 
Further note if $\gamma\in[0,1)$, $M\in(0,\infty)$ and 
the transition kernel is $L_1$-exponentially 
convergent with $(\gamma,M)$ then we have, using
\eqref{novak_rudolf eq: diagramm}, that ${\rm gap}(P)\geq 1-\gamma$. 

Now we ask how $e_\nu(S_{n,n_0},f)$ 
behaves depending on the initial distribution.
The idea is to decompose the error in a suitable way. 
For example in a bias and variance term. 
However, we want to have an
estimate with respect to $\left \Vert f \right \Vert_{2}$ and in 
this setting the following decomposition
is more convenient:
\[
e_\nu(S_{n,n_0},f)^2 = e_\pi(S_{n,n_0},f)^2 + \mbox{rest},
\]
where rest denotes an additional 
term such that equality holds. Then, we estimate
the remainder term and use 
Lemma~\ref{novak_rudolf thm: err_bound_stat} to obtain an error bound.
For further details of the proof of the following error bound we refer to 
\cite[Theorem~3.34 and Theorem~3.41]{novak rudolf:Ru12}.

\begin{theorem}  \label{novak_rudolf thm: err_bound}
Let $(X_n)_{n\in\mathbb{N}}$ be a Markov chain with reversible 
transition kernel $K$
and initial distribution $\nu$.
Further, let
  \begin{equation*}
    \Lambda = \sup\{ \alpha \colon \alpha \in {\rm {s}pec}(P-S) \},
  \end{equation*}
where ${\rm {s}pec}(P-S)$ denotes 
the spectrum of the operator $P-S\colon L_2 \to L_2$,  and
assume that $\Lambda<1$.
Then
  \begin{equation} \label{novak_rudolf eq: expl_err_bound}
    \sup_{\left \Vert f \right \Vert_{p}\leq 1 } e_\nu(S_{n,n_0},f)^2  
\leq \frac{2}{n(1-\Lambda)} + \frac{2\, C_\nu \gamma^{n_0}}{n^2(1-\gamma)^2}
  \end{equation}
holds for $p=2$ and for $p=4$ under the following conditions
 \begin{enumerate}
    \item for $p=2$, $\frac{d\nu}{d\pi}\in L_\infty$ and a transition kernel
    $K$ which is $L_1$-exponentially convergent with $(\gamma,M)$ where
    $ 
    C_\nu = M \left \Vert \frac{d\nu}{d\pi}-1\right \Vert_{\infty};
    $
\item \label{novak_rudolf en: p_4} for $p=4$, 
$\frac{d\nu}{d\pi}\in L_2$ and $1-\gamma={\rm gap}(P)>0$ where
    $
    C_\nu = 64 \left \Vert \frac{d\nu}{d\pi}-1\right \Vert_{2}.
    $
  \end{enumerate}
\end{theorem}
Let us discuss the results. 
If the transition kernel is $L_1$-exponentially ergodic, then we
have an explicit error bound for integrands $f\in L_2$ whenever the initial 
distribution has a density $\frac{d \nu}{ d \pi} \in L_\infty$.
However, in general it is difficult to provide 
explicit values $\gamma$ and $M$ such that
the transition kernel is $L_1$-exponentially 
convergent with $(\gamma,M)$. This motivates 
to consider transition kernel which satisfy a 
weaker convergence property, such as the existence
of an $L_2$-spectral gap. In this case we have an 
explicit error bound for integrands $f\in L_4$ 
whenever the initial distribution has a 
density $\frac{d \nu}{ d \pi} \in L_2$. Thus,
by assuming a weaker convergence property of 
the transition kernel we obtain a weaker result in the sense
that $f$ must be in $L_4$ rather than $L_2$. However,
with respect to $\frac{d \nu}{d \pi}$ we do not need boundedness anymore, 
it is enough that $\frac{d \nu}{d \pi} \in L_2$.
 
In Theorem~\ref{novak_rudolf thm: err_bound} we provided explicit error bounds 
and we add in passing that also other error bounds are known, 
see 
\cite{novak rudolf:BeCh09,novak rudolf:JoOl10,novak rudolf:LaMiNi11,novak rudolf:Ru12}.

If we want to have an error of $\varepsilon \in(0,1)$ it is still not
clear how to choose $n$ and $n_0$ to minimize 
the total amount of steps $n+n_0$. 
How should we choose 
the burn-in $n_0$? Let $e(n,n_0)$ be the right hand side
of \eqref{novak_rudolf eq: expl_err_bound} and assume that $\Lambda= \gamma$. 
Further, assume that we have computational resources for $N=n+n_0$
steps of the Markov chain. We want to get an $n_{\text{opt}}$ 
which minimizes $e(N-n_0,n_0)$.
In \cite[Lemma~2.26]{novak rudolf:Ru12} the following is proven: 
For all $\delta>0$
and large enough $N$ and $C_\nu$ the number   
$n_{\text{opt}}$ satisfies
\[
n_{\text{opt}} \in \left[ \frac{\log C_\nu}{\log\gamma^{-1}}, 
(1+\delta)\frac{\log C_\nu}{\log\gamma^{-1}} \right].
\]  
Further note that $\log \gamma^{-1} \geq 1-\gamma$.
Thus, in this setting $ n_{\text{opt}} 
= \lceil\frac{\log C_\nu}{1-\gamma} \rceil$ 
is a reasonable and almost optimal choice for the burn-in.

\section{Application of the error bound 
and limitations of MCMC} 

First, we briefly introduce a technique 
to prove a lower bound of the spectral gap
if the Markov operator of a transition kernel 
is positive semidefinite on $L_2$. 
The following result, known 
as \emph{Cheeger's inequality}\index{Cheeger's inequality}, 
is in this form due to Lawler and Sokal \cite{novak rudolf:LaSo88}.

\begin{proposition} \label{novak_rudolf prop: cheeger}
Let $K$ be a reversible transition kernel, 
which induces a Markov operator $P\colon L_2 \to L_2$. 
Then 
\[
\frac{\varphi^2}{2} \leq1-\Lambda \leq 2 \varphi,
\]
where $  \Lambda = \sup\{ \alpha \colon \alpha \in {\rm {s}pec}(P-S) \} $
and
\[
\varphi = \inf_{0<\pi(A)\leq 1/2} 
\frac{\int_A K(x,A^c) \, \pi({\rm d} x)}{\pi(A)}
\]
is the conductance of $K$.
\end{proposition}

Now we state different applications 
of Theorem~\ref{novak_rudolf thm: err_bound}.

\subsection{Hit-and-run algorithm\index{Hit-and-run algorithm}}
\label{novak_rudolf sec: har}
We consider the example of Section~\ref{novak_rudolf sec: motivation}.
Let $G \in \mathcal{G}_{r,d}$, 
see \eqref{novak_rudolf classGrd}, and let $\mu_G$ be the uniform 
distribution in $G$. We define 
\begin{equation}   \label{novak_rudolf Frd} 
 \mathcal{F}_{r,d} = \{ (f,G)\colon G \in \mathcal{G}_{r,d},
\,  f\in L_4(\mu_G),\, \left \Vert f \right \Vert _{4}\leq 1\}.
\end{equation} 
The goal is to approximate
\[
S(f,\mathbf{1}_G)  = \frac{1}{{\rm  vol}_d(G)} \int_G f(x)\, \text{\rm d} x,
\] 
where $(f,G) \in \mathcal{F}_{r,d}$. The hit-and-run algorithm 
defines
a Markov chain which satisfies the assumptions 
of Theorem~\ref{novak_rudolf thm: err_bound}.
A step from $x\in G$ of the hit-and-run algorithm works as follows

\begin{enumerate}
\item 
Choose a direction, say $\theta$, uniformly distributed 
on the sphere $\partial B_d$.
\item 
Choose the next state, say $y \in G$, uniformly 
distributed in $G\cap \{ x+\theta r\colon r\in \mathbb{R} \}$.
\end{enumerate}

After choosing a direction $\theta$ one samples the 
next state $y\in G$ with respect to the uniform 
distribution in the line determined by the current 
state $x$ and the direction $\theta$ restricted to $G$.
The random number, say $u \in [0,1]$, for the second part 
is chosen independently of the first part and also all steps 
are independent. 

Lova\'sz and Vempala prove in \cite[Theorem 4.2, p. 993]{novak rudolf:LoVe06} 
a lower bound of the conductance $\varphi$,
see Proposition~\ref{novak_rudolf prop: cheeger} 
for the definition of the conductance.

\begin{proposition}
Let $G\in \mathcal{G}_{r,d}$. Then, the conductance of the 
hit-and-run algorithm is bounded from below by $2^{-25}(dr)^{-1}$.
\end{proposition}
It is known that the hit-and-run algorithm induces 
a positive semidefinite Markov operator, 
say $H$, see \cite{novak rudolf:RuUl12}.
By Proposition~\ref{novak_rudolf prop: cheeger} we obtain 
\[
{\rm gap}(H) \geq \frac{2^{-51}}{(dr)^2}
\]
and Theorem~\ref{novak_rudolf thm: err_bound} implies the following error bound
for the class $\mathcal{F}_{r,d}$, 
see~\eqref{novak_rudolf classGrd} and~\eqref{novak_rudolf Frd}.  

\begin{theorem}  \label{novak_rudolf thm: har}
Let $\nu$ be the uniform distribution on $B_d$. Let $(X_n)_{n\in\mathbb{N}}$ 
be a Markov chain with transition kernel,
  given by the hit-and-run algorithm, and initial distribution $\nu$. Let 
  \[
   n_0 = \lceil 4.51\cdot 10^{15} d^2 r^2 ( d \log r + 4.16) \rceil.
  \]
  Then
\[
\sup_{(f,G) \in \mathcal{F}_{r,d}} e_\nu(S_{n,n_0},(f,\mathbf{1}_G)) 
\leq 9.5\cdot 10^7 \frac{d r}{\sqrt{n}} + 6.4\cdot 10^{15} \frac{d^2 r^2}{n}.
\]
\end{theorem}   

This result states that the number of oracle calls 
for $f$ and $G$ to obtain 
an error $\varepsilon>0$
is bounded by
$
\kappa\, d^2 r^2 (\varepsilon ^{-2} + d\log r),
$
for an explicit constant $\kappa>0$.
Hence 
the computation of $S(f,\mathbf{1}_G)$ on the class
$\mathcal{F}_{r,d}$ is polynomially tractable, 
see \cite{novak rudolf:NoWo08,novak rudolf:NoWo10,novak rudolf:NoWo12}.
The tractability result can be extended also
to other classes of functions, see \cite{novak rudolf:Ru13}.
Note that we applied the second 
statement of Theorem~\ref{novak_rudolf thm: err_bound}.  
It is known that the hit-and-run algorithm is $L_1$-exponentially 
ergodic with $(\gamma,M)$, for
some $\gamma\in(0,1)$ and $M\in(0,\infty)$. But the best known
numbers $\gamma$ and $M$ are exponentially bad in terms of the dimension, see
\cite{novak rudolf:Sm84}.

\subsection{Metropolis-Hastings algorithm\index{Metropolis-Hastings algorithm}}

Let $G \subset \mathbb{R}^d$ and $\rho \colon G \to (0,\infty)$, 
where $\rho$ is integrable with respect to the Lebesgue measure.
We define the distribution $\pi_\rho$ on $(G,\mathcal{B}(G))$ by
\[
\pi_\rho(A) = \frac{\int_A \rho(x)\, {\rm d} x}{\int_G \rho(x)\, 
{\rm d} x}, \qquad A\in \mathcal{B}(G).
\]
The goal is to compute 
\[
 S(f,\rho) = \int_G f(x)\, \pi_\rho({\rm d}x) 
= \frac{\int_G f(x) \rho(x)\, {\rm d}x }{ \int_G \rho(x)\,{\rm d}x} 
\]
for functions $f\colon G \to \mathbb{R} $ which are 
integrable with respect to $\pi_\rho$.

The \emph{Metropolis-Hastings algorithm}
defines a Markov chain which approximates $\pi_\rho$. 
We need some further notations.
Let $q\colon G\times G \to [0,\infty]$ be a function such that
$q(x,\cdot)$ is Lebesgue integrable for all 
$x\in G$ with $\int_G q(x,y)\,{\rm d} y\leq 1$.
Then 
\[
Q(x,A)= \int_A q(x,y)\,{\rm d} y + \mathbf{1}_A(x)\left(1-\int_G q(x,y)\,
{\rm d} y\right),
\quad x\in G,\; A\in\mathcal{B}(G),
\]
is a transition kernel and we call $q(\cdot,\cdot)$ \emph{transition density}. 
The idea is to modify $Q$, 
such that $\pi_\rho$ gets a stationary distribution
of the modification.
We propose a state with $Q$ and with a certain probability, 
which depends on $\rho$, 
the state is accepted.
Let $\alpha(x,y)$ be the acceptance probability 
\[
  \alpha(x,y) = \begin{cases}
	      1 & \mbox{if } q(x,y)\rho(x)=0,\\
     \min\{ 1 , \frac{q(y,x)\rho(y)}{q(x,y)\rho(x)}\} & \mbox{otherwise}.
            \end{cases}
\]
The transition kernel 
of the Metropolis-Hastings algorithm is
\begin{align*}
  K_\rho(x,A) 
& = \int_A \alpha(x,y)\, q(x,y) {\rm d} y + 
\mathbf{1}_A(x) \left[ 1-\int_G \alpha(x,y)\,q(x, y) {\rm d} y\right]
\end{align*}
for $x\in G$ and  $A\in \mathcal{B}(G)$.
The transition kernel $K_\rho$ is reversible with respect to $\pi_\rho$.
{}From the current state $x\in G$ a single transition
of the 
algorithm works as follows:
\begin{enumerate}
 \item Sample a proposal state $y\in G$ with respect to $Q(x,\cdot)$.
 \item With probability $\alpha(x,y)$ return $y$,
otherwise reject $y$ and return $x$.
\end{enumerate}

Again, all steps are done independently of each other. 
If $q(x,y)=q(y,x)$, i.e.
$q$ is symmetric, then $K_\rho$ is called \emph{Metropolis algorithm} 
and if $q(x,y)=\eta(y)$ 
for a function $\eta\colon G\to (0,\infty)$ for all $x,y\in G$, then
$K_\rho$ is called \emph{independent Metropolis algorithm}.\\

Let $G\subset \mathbb{R}^d$ be bounded 
and for $C\geq 1$ let 
\begin{equation} \label{novak_rudolf R_C} 
\mathcal{R}_C = \{ \rho\colon G \to (0,\infty) \mid 1\leq \rho(x) \leq C  \}. 
\end{equation} 
Thus, for any $\rho \in \mathcal{R}_C$ 
holds $\sup \rho/ \inf \rho \leq C$. 
If $\rho\colon G \to (0,\infty)$ satisfies 
$\sup \rho/ \inf \rho \leq C$, then
\[
\frac{\left \Vert \rho \right \Vert_{\infty}}{C} \leq \rho(x) \leq C \inf \rho.
\]
Thus, $  C\cdot\rho /\left \Vert \rho\right \Vert_{\infty} \in \mathcal{R}_C$.
 We consider an independent Metropolis algorithm.
 The proposal transition kernel is 
 \[
Q(x,A)= \mu_G (A) = \frac{{\rm vol}_d(A)}{{\rm vol}_d(G)}, 
\quad A\in\mathcal{B}(G),
 \]
i.e. a state is proposed with the uniform distribution in $G$. 
Then 
\[
K_\rho(x,A) = \int_A \alpha(x,y) \frac{{\rm d} y}{{\rm vol}_d(G)} 
+ \mathbf{1}_A(x) \left(1-\int_G \alpha(x,y)\, 
\frac{{\rm d} y}{{\rm vol}_d(G)}\right),
\]
where $\alpha(x,y)= \min\{ 1,\rho(y)/\rho(x) \}$.
The transition operator 
$P_\rho\colon L_2(\pi_\rho) \to L_2(\pi_\rho)$, 
induced by $K_\rho$, is positive semidefinite. 
For details we refer to \cite{novak rudolf:RuUl12}. 
Thus, ${\rm gap}(P_\rho)=1-\Lambda_\rho$, with $\Lambda_{\rho} = \Lambda$.
Further, for $\rho \in \mathcal{R}_C$
Theorem~2.1 of \cite{novak rudolf:MeTw96} 
provides a criterion for uniform ergodicity of 
the independent Metropolis algorithm. 
Namely, $K_\rho$ is uniformly ergodic with $(\gamma,1)$ for
$\gamma = 1-C^{-1}/{\rm vol}_d (G)$. 
Thus, by \eqref{novak_rudolf eq: diagramm} we have that 
it is $L_1$-exponentially ergodic with $(\gamma,2)$. 
Further, by \eqref{novak_rudolf eq: diagramm}
we obtain
\[
1-\Lambda_\rho={\rm gap}(P_\rho) \geq \frac{C^{-1}}{{\rm vol}_d (G) }. 
\]
Let 
\begin{equation}   \label{novak_rudolf F_C,d} 
\mathcal{F}_{C,d} = \{ (f,\rho) \colon \rho \in \mathcal{R}_{C},\, 
f\in L_2(\pi_\rho),\, \left \Vert f\right \Vert_{2}\leq 1	\}.
\end{equation}  
We apply Theorem~\ref{novak_rudolf thm: err_bound} and obtain for 
the class 
$\mathcal{F}_{C,d}$
(see \eqref{novak_rudolf R_C} and \eqref{novak_rudolf F_C,d}) 

\begin{theorem}     \label{novak_rudolf TH3} 
Let $(X_n)_{n\in\mathbb{N}}$ 
be a Markov chain with transition kernel,
given by the Metropolis algorithm with proposal 
$\mu_G$, and initial distribution $\mu_G$. 
Let 
\[
n_0 
= \left\lceil C {\rm vol}_d(G)\log(2C)
\right \rceil.
 \]
Then
\[
\sup_{(f,\rho) \in \mathcal{F}_{C,d}} e_\nu(S_{n,n_0},(f,\rho))^2 
\leq 
\frac{2\, C\, {\rm vol}_d(G) }{n} + \frac{4\, C^2\, {\rm vol}_d(G)^2 }{n^2}.
\]
\end{theorem} 

The upper bound in Theorem~\ref{novak_rudolf TH3} 
does not depend on the dimension $d$, as long as 
${\rm vol}_d(G)$ and $C$ do not depend on $d$. 
In some applications, however, the upper bound is rather useless 
since $C=C_d$ is exponentially large in $d$. 
Assume, for example, that  
\begin{equation} \label{novak_rudolf eq: dens_normal}
\rho(x)= \exp(-\alpha \vert x \vert ^2),
\end{equation}
i.e.
$\rho$ is the non-normalized density of 
a $N(0, \sqrt{2\alpha^{-1}})$ random variable. 
We consider scaled versions of $\rho$.
If 
$G=B_d $, then $\exp(\alpha)\rho \in 
\mathcal{R}_{\exp(\alpha)}$ and if $G=[-1,1]^d$, 
then $\exp(\alpha d)\rho \in \mathcal{R}_{\exp(\alpha d)}$.  
This is bad, since $C$, for example $\exp(\alpha)$ or $\exp(\alpha d)$, might
depend exponentially on $\alpha$ and $d$.

This example shows that we would greatly prefer an upper bound 
where $C$ is replaced by a power of $\log C$. 
However, on the class $\mathcal{F}_{C,d}$ this is not possible. 
The same proof as 
in \cite[Theorem~1]{novak rudolf:MaNo07} leads to the following
lower bound for \emph{all} randomized algorithms. 

\begin{theorem} 
Any randomized algorithm $S_n$ that uses $n$ values of $f$ and $\rho$ 
satisfies the lower bound
\[
\sup_{(f,\rho) \in \mathcal{F}_{C,d}} e(S_n,(f,\rho)) \geq \frac{\sqrt{2}}{6}
\begin{cases}
\sqrt{\frac{C}{2n}} & 2n \geq C-1,\\
\frac{3C}{C+2n-1} & 2n < C-1.
\end{cases}
\]
\end{theorem}

The class $\mathcal{F}_{C,d}$ is too large. 
Thus the error bound is not satisfying.
In the following we prove a much better upper bound 
for a smaller class of densities.
Let $G=B_d$ and let $\rho$ be log-concave, i.e. for all $\lambda \in (0,1)$
and for all $x,y\in B_d$ we have
\begin{equation} \label{novak_rudolf eq: log_conc}
\rho(\lambda x + (1-\lambda) y) \geq \rho(x)^\lambda \rho(y)^{1-\lambda}.
\end{equation}
Then let
\begin{equation}   \label{novak_rudolf Rad} 
\mathcal{R}_{\alpha,d} = \{ \rho \colon B_d \to (0,\infty)\mid   \rho\; 
\mbox{is log-concave},\, 
\vert \log\rho(x)-\log\rho(y) \vert  \leq \alpha \vert x-y \vert \}. 
\end{equation}  
We consider log-concave densities 
where $\log \rho$ is Lipschitz continuous with constant $\alpha$.
Note that the setting is more restrictive compared to the previous one.
The goal is to get an upper error bound which 
is polynomially in $\alpha$ and $d$.
We consider a \emph{Metropolis algorithm based on a ball walk}. 
For $\delta >0$ the transition
kernel of the $\delta$ ball walk is
\[
B_\delta(x,A) 
= \frac{{\rm vol}_d(A\cap B_\delta(x))}{{\rm vol}_d(B_\delta(0))} 
+ \mathbf{1}_A(x)
\left( 1- \frac{{\rm vol}_d(G\cap B_\delta(x))}{{\rm vol}_d(B_\delta(0))}
\right), 
      \quad x\in G,\,A\in \mathcal{B}(G),
    \]
where $B_\delta(x)$ denotes the Euclidean ball with radius $\delta$ around $x$.
Let $K_{\rho, \delta}$
be the transition kernel of the Metropolis 
algorithm with ball walk proposal $B_\delta$, 
let $P_{\rho,\delta}$ be the corresponding transition operator and let 
$\Lambda_{\rho,\delta}$ be the largest element of the spectrum of 
$P_{\rho,\delta}-S\colon L_{2}(\pi_\rho) \to L_{2}(\pi_\rho)$.

In \cite[Corollary~1]{novak rudolf:MaNo07} the following result is proven.
\begin{proposition} \label{novak_rudolf prop: metro_bw_conduct}
    Let $\rho\in \mathcal{R}_{\alpha,d}$ and let 
    $\delta=\min\{1/\sqrt{d+1},{\alpha^{-1}}\}$. 
    Then, the conductance of $K_{\rho, \delta}$ is bounded from below by
    \[
\frac{0.0025}{\sqrt{d+1}} 
\min\left\{\frac{1}{\sqrt{d+1}},\frac{1}{\alpha}\right\}.
    \]
  \end{proposition}
By Proposition~\ref{novak_rudolf prop: cheeger} and 
Proposition~\ref{novak_rudolf prop: metro_bw_conduct} we have 
a lower bound of $1-\Lambda_{\rho,\delta}$. 
However, to apply Theorem~\ref{novak_rudolf thm: err_bound} we need 
a lower bound on ${\rm gap}(P_{\rho,\delta})$. 
Let $\widetilde{K}_{\rho,\delta}$ be the transition kernel of the
lazy version of $K_{\rho, \delta}$, i.e. 
for $x\in G$ and $A\in \mathcal{B}(G)$ holds
$
\widetilde{K}_{\rho,\delta}(x,A) = 
(K_{\rho, \delta}(x,A) + \mathbf{1}_{A}(x) )/2. 
$
In words, $\widetilde{K}_{\rho,\delta}$ can be described as follows: 
With probability $1/2$ stay at the current state and 
with with probability $1/2$ do 
one step with $K_{\rho, \delta}$.
This transition kernel induces a positive semidefinite operator 
$\widetilde{P}_{\rho,\delta} \colon L_{2}(\pi_\rho) \to L_{2}(\pi_\rho)$ with
\[
{\rm gap}(\widetilde{P}_{\rho,\delta}) 
= \frac{1}{2}(1 +  \Lambda_{\rho,\delta}).
\]
Let
\begin{equation}   \label{novak_rudolf Falphad} 
 \mathcal{F}_{\alpha,d} = 
\{ (f,\rho) \colon \rho \in \mathcal{R}_{\alpha,d},\, 
f\in L_4(\pi_\rho),\, \left \Vert f \right \Vert_{4}\leq 1	\},
\end{equation}\
and recall that $\mathcal{R}_{\alpha,d}$ is 
defined in \eqref{novak_rudolf Rad}.
Note that we assumed $G=B_d$. 
Now we can apply Theorem~\ref{novak_rudolf thm: err_bound} for the 
lazy Metropolis algorithm with ball 
walk proposal $\widetilde{K}_{\rho,\delta}$.

\begin{theorem}  \label{novak_rudolf thm: metro_bw_mse} 
Let $\nu$ be the uniform distribution on $B_d$ and 
let us assmue that $\delta=\min\{1/\sqrt{d+1},\alpha^{-1}\}$.
Let $(X_n)_{n\in\mathbb{N}}$ 
be a Markov chain with transition kernel $\widetilde{K}_{\rho,\delta}$, i.e.
the lazy version of the Metropolis algorithm 
with ball walk proposal $B_\delta$,
and initial distribution $\nu$. Let 
\[
n_0 = \lceil 5.92\cdot 10^{6} (d+1) \max\{ \alpha^2,d+1 \} ( 2\alpha + 4.16) 
\rceil.
\]
Then
\begin{align*}
\sup_{(f,G) \in \mathcal{F}_{\alpha,d}} e_\nu(S_{n,n_0},(f,\rho)) 
& \leq 1089 \frac{\sqrt{d+1} \max\{ \alpha, \sqrt{d+1}\} }{\sqrt{n}} \\
& \qquad+ 8.38 \cdot 10^5 \frac{(d+1) \max\{ \alpha^2, d+1\} }{n} .
   \end{align*}
\end{theorem} 
The last theorem states that the number 
of oracle calls of $f$ and $\rho$ to obtain
an error $\varepsilon>0$ is bounded 
by $\kappa\, d \max\{ \alpha^2,d \}(\varepsilon^2+\alpha)$. Hence 
the computation of $S(f,\rho)$ is polynomially tractable.
Note that $\mathcal{R}_{\alpha,d}$ might be interpreted 
as a subclass of $\mathcal{R}_C$ with $C=\exp (2\alpha)$ and $G=B_d$, since
$
\rho\in \mathcal{R}_{\alpha,d}$ 
implies
$
\exp(2\alpha) \rho/\left \Vert \rho \right \Vert_{\infty}  
\in \mathcal{R}_{\exp (2\alpha)}. 
$
Thus, by Theorem~\ref{novak_rudolf thm: metro_bw_mse} 
we obtain that the number of oracle calls
to get an error $\varepsilon$ also
depends polynomially on $\log C$, since $C=\exp(2\alpha)$. 

\section{Open problems and related comments}

\begin{itemize} 
\item
We do not know whether an error bound 
as in Theorem~\ref{novak_rudolf thm: err_bound} 
holds for 
$f \in L_2$ if ${\rm gap} (P)>0$. 
\item 
In \cite{novak rudolf:RuSc13} error 
bounds of $S_{n,n_0}$ for $f\in L_p$ with $1 < p \leq 2$ are proven.
Then 
one needs a new error criterion, here 
the absolute mean error
\[
\mathbb{E}_{\nu,K} 
\vert S_{n,n_0}(f)-S(f) \vert
\]
is used. 
If the Markov chain is $L_1$-exponentially convergent, 
then the error bound decreases with $n^{1/p-1}$. 
For a Markov chain with $L_2$-spectral gap a similar error bound is shown.
\item
The tractability results in Theorem~\ref{novak_rudolf thm: har} and
Theorem~\ref{novak_rudolf thm: metro_bw_mse} are nice since the degree 
of the polynomial is small. 
Nevertheless, the upper bound is not really useful 
because of the huge constants. 
Is it possible to prove these or similar results with much smaller 
constants?

\item A related question would be the construction
of Markov chain quasi-Monte Carlo 
methods, see \cite{novak rudolf:ChDiOw11,novak rudolf:DiRuZh13}.
Here the idea is to derandomize the Markov chain by using 
a carefully constructed deterministic sequence of
numbers to obtain a sample $x_1,\dots,x_{n+n_0}$.
However, explicit constructions 
with small
error bounds are not known.
\end{itemize}



%
%
%

\end{document}